\documentclass[a4paper,12pt]{amsart}

\usepackage{amsmath,amssymb,amsthm}    
\usepackage{comment}
\usepackage{ulem}
\usepackage{tikz}
\usepackage{mathdots}
\usepackage{tikz}
\usepackage{subcaption}
\usepackage[subrefformat=parens,labelformat=parens]{subcaption}
\usetikzlibrary{calc,positioning,decorations.pathmorphing,decorations.pathreplacing}
\tikzset{emp/.style={double distance = 0.3ex}}

\tikzset{M edge/.style={line width=1.3pt,double distance=1.1pt}}
\tikzset{F1 edge/.style={line width=1.3,color=red,->}}
\tikzset{F2 edge/.style={line width=1.3,color=blue,->}}
\tikzset{E edge/.style={line width=1.3,color=black,-}}

\tikzset{squared black vertex/.style={draw,minimum size=2mm,inner sep=0pt,outer sep=3pt,fill=black, color=black}}
\tikzset{red vertex/.style={circle,draw,minimum size=2mm,inner sep=0pt,outer sep=2pt,fill=red, color=red}}
\tikzset{blue vertex/.style={circle,draw,minimum size=2mm,inner sep=0pt,outer sep=2pt,fill=blue, color=blue}}
\tikzset{black vertex/.style={circle,draw,minimum size=2mm,inner sep=0pt,outer sep=2pt,fill=black, color=black}}
\tikzset{small black vertex/.style={circle,draw,minimum size=1.2mm,inner sep=0pt,outer sep=1.2pt,fill=black, color=black}}
\tikzset{small white vertex/.style={circle,draw,minimum size=1.2mm,inner sep=0pt,outer sep=1.2pt,color=black,fill=white}}
\tikzset{square vertex/.style={draw,minimum size=1.2mm,inner sep=0pt,outer sep=1.2pt,fill=black, color=red}}
\tikzset{white vertex/.style={circle,draw,minimum size=2mm,inner sep=0pt,outer sep=3pt,color=black,fill=white}}

\tikzset{fatpath/.style={line width=9pt,rounded corners=.1mm}}

\tikzstyle{edge}=[line width=1.3]

\tikzstyle{color1}=[color=blue] 
\tikzstyle{color2}=[color=red]
\tikzstyle{color3}=[color=green] 
\tikzstyle{color4}=[fill=yellow]
\tikzstyle{color5}=[ dashed] 

\tikzstyle{backcolor1}=[color=gray!55!white] 
\tikzstyle{backcolor2}=[color=blue!35!white] 

\usepackage{standalone}
\usepackage{mathtools}
\DeclarePairedDelimiter\ceil{\lceil}{\rceil}

\newcommand{\etal}{\textit{et~al.}}

\usepackage{enumitem}

\usepackage{amsaddr}

\usepackage[T1]{fontenc}
\usepackage[english]{babel}

\usepackage{bm}

\usepackage{fullpage}

\newtheorem{theorem}             {Theorem}
        
\newtheorem{conjecture}	[theorem] {Conjecture}

\newtheorem{proposition}[theorem] {Proposition}   
\newtheorem{corollary}	[theorem] {Corollary}
     
\newtheorem{claim}	[theorem] {Claim}

\title{On two conjectures about the intersection of longest paths and cycles.} 
\thanks{
  J. Gutiérrez was partially supported by
by Movilizaciones para Investigación AmSud, PLANarity and distance IN Graph theory (E070-2021-01-Nro.6997) and Fondo Semilla UTEC 871075-2022.  
  }

\author{Juan Gutiérrez} 
\address{\vspace{-4mm}Departamento de Ciencia de la Computación\\ 
Universidad de Ingeniería y Tecnología (UTEC), Perú}
\email{jgutierreza@utec.edu.pe}

\author{Christian Valqui}
\address{Pontificia Universidad Cat\'olica del Per\'u, Secci\'on Matem\'aticas, PUCP, Av. Universitaria 1801, San Miguel, Lima 32, Per\'u}
\email{cvalqui@pucp.edu.pe}

\usepackage{lineno}

\begin{document}
\normalem

\maketitle

\begin{abstract}
A conjecture attributed to Smith states that every pair of longest cycles in a $k$-connected graph intersect each other in at least $k$ vertices.
In this paper, we show that every pair of longest cycles in a~$k$-connected graph on $n$ vertices intersect each other in at least~$\min\{n,8k-n-16\}$ vertices,
which confirms Smith's conjecture when $k\geq  (n+16)/7$. 
An analog conjecture for paths instead of cycles was stated by Hippchen. 
By a simple reduction, we relate both conjectures, showing that Hippchen's conjecture is valid
when either $k \leq 6$ or $k \geq (n+9)/7$.
\end{abstract}

\section{Introduction and Preliminaries}

%In this paper all graphs considered are simple and the notation and terminology are standard.
It is known that every pair of longest cycles (paths)
in a 2-connected (connected) graph intersect each other in at least two vertices (one vertex).
A conjecture proposed by Gr\"otschel, and attributed to Scott Smith \cite[Conjecture 5.2]{ Grotschel84}, states that, in a $k$-connected graph, with $k\geq 2$, every pair of longest cycles intersect each other in at least $k$ vertices.
Years later, Hippchen \cite[Conjecture 2.2.4]{Hippchen08} conjectured that, for~$k$-connected
graphs, every pair of longest paths intersect each other in at least~$k$ vertices.

Smith's conjecture has been verified up to~$k=7$
\cite{Grotschel84}, and, for a general~$k$, it was proved that every pair of longest cycles intersect in at least~$ck^{3/5}$ vertices, for a constant~$c \thickapprox 0.2615$
\cite{Chen98}.
For Hippchen's conjecture, the case~$k=3$ was proved by Hippchen himself
\cite[Lemma 2.2.3]{Hippchen08}.
Later, the first author show it for $k=4$,
and recently, Cho \etal \cite{Cho2022} show it for $k=5$.

As we can observe in the literature, efforts on both conjectures has been made independently. The first main contribution of this paper is relate these conjectures. We will show that if
Smith's conjecture is valid for a fixed $k$, then
Hippchen's conjecture is valid for $k-1$.
This relationship implies that efforts 
can be concentrated on Smith's conjecture.
So, as Smith's conjecture is known to be valid for $k\leq 7$, an easy corollary is that Hippchen's conjecture is true for~$k\leq 6$.
(For an independent proof of this fact, see
\cite{Gutierrez2021-bitraceable}).

The second main contribution of this paper is
an improvement of a result on Hippchen's conjecture when $k=\Omega(n)$, where $n$ is the number of vertices of the graph.
In \cite{Gutierrez2021-Pro}, the first author
showed that Hippchen's conjecture is valid when
$k \geq (n-2)/3$. This result was improved
by Cho \etal \cite{Cho2022}. They showed
that Hippchen's conjecture is valid when
$k \geq (n+2)/5$.
In this paper, we improve this result, by showing
that Hippchen's conjecture is valid when
$k \geq (n+9)/7$.
In fact, we show the stronger statement
that every pair of longest paths in a $k$-connected graph intersect in at least $\min\{k,8k-n-9\}$ vertices.
Moreover, this result is a consequence of an analog result on cycles: we show that every pair of longest cycles intersect in at least $\min\{k,8k-n-16\}$ vertices.

In this paper all graphs are simple (without loops or parallel edges) and the notation and terminology are standard.
We also consider simple paths and cycles, that is, repetitions of vertices or edges is not allowed.
The \textbf{length} of a path~$P$ (cycle $C$) is the number of edges it has, and it is denoted by~$|P|$ ($|C|$). A \textbf{longest path} (cycle) in a graph is a path (cycle) with maximum length among all paths (cycles).
Given a path~$P$ and two vertices~$x$ and~$y$ 
in~$P$, we denote by~$P[x,y]$ the subpath of~$P$ with extremes~$x$ and~$y$.

Given two set of vertices~$S$ and~$T$ in a graph~$G$, an~$\bm{S}$-$\bm{T}$ \textbf{path} is a path with one end in~$S$, the other end in~$T$, and whose internal vertices are neither in~$S$ nor~$T$. If~$S=\{v\}$, we also say that an~$S$-$T$ path is a~$v$-$T$ path. When we refer to the intersection of two paths or cycles in a graph, we mean vertex-intersection, that is, the set of vertices they share.
Two paths are \textbf{internally disjoint} if
they have no internal vertices in common.

A graph~$G$ is~$\bm{k}$\textbf{-connected} if, for any two distinct 
vertices~$u$ and~$v$ in~$G$, there exists a set of
$k$~$u$-$v$ internally disjoint paths.
It is easy to see that for a~$k$-connected graph on~$n$ vertices, we have~$k\leq n-1$.
A set $S \subseteq V(G)$ of a connected graph $G$ is a \textbf{separator} if $G-S$ has more than one component.
It is known that a graph is $k$-connected if and only if
every separator in the graph has size at least $k$.

\section{Two families of conjectures on longest paths and cycles}

For every $k\geq 1$ and every $r \leq k$, we consider the following families of conjectures.

\begin{conjecture}[Conjecture P$\mathbf{(k,r)}$]\label{conj:Pkr}
In any $k$-connected graph, every pair of longest paths intersect in at least~$r$ vertices.
\end{conjecture}

\begin{conjecture}[Conjecture C$\mathbf{(k,r)}$]\label{conj:Ckr}
In any $k$-connected graph, every pair of longest cycles intersect in at least~$r$ vertices.
\end{conjecture}

Note that Conjecture P$\mathbf{(k,k)}$
is Hippchen's Conjecture \cite{Hippchen08}, and Conjecture C$\mathbf{(k,k)}$
is Smith's Conjecture \cite[Conjecture 5.2]{ Grotschel84}. The first main purpose of this paper is to relate Conjecture P$\mathbf{(k+1,r+1)}$ and 
Conjecture C$\mathbf{(k,r)}$, as in the next Theorem.

\begin{theorem}\label{thm:CimpliesP}
For every $k\geq 1$ and every $r\leq k$,
Conjecture $C(k+1,r+1)$ implies Conjecture $P(k,r)$.
\end{theorem}
\begin{proof}
Let $G$ be a $k$-connected graph on $n$ vertices.
Let $P$ and $Q$ be two longest paths in $G$.
We construct a new graph, $G'$, by adding a new vertex $s$ to $G$ and joining it to every vertex of $G$. 

We will show that $G'$ is $(k+1)$-connected.
Suppose by contradiction that it is not the case. Then $G'$ has a separator $X'$ with size at most $k$. As $s$ is universal in $G'$,
we must have $s\in X'$. Indeed, otherwise,
any pair of vertices of $G'$ different from $s$ are joined by $s$.
As $X'$ is a separator of $G'$, there exists two distinct vertices $u$ and $v$ in $G'$ such that there is no $uv$-path in $G'-X'$.
As $G$ is a subgraph of $G'$, there is also no
$uv$-path in $G-X'$, which implies
that $X' \cap V(G)$ is a separator in $G$.
As $s \in X'$, $|X' \cap V(G)| \leq k-1$, which contradicts the fact that $G$ is $k$-connected.

Let $P'$ and $Q'$ be the cycles in $G'$ with vertex set $P \cup \{s\}$ and $Q \cup \{s\}$,
respectively. If Conjecture $C(k+1,r+1)$
is true, then $P'$ and $Q'$ intersect in at least $r+1$ vertices, thus $P$ and $Q$ intersect in at least $r$ vertices.
\end{proof}

Conjecture P$\mathbf{(k,k)}$ has been proved
for $k=3$ \cite{Hippchen08}, $k=4$ \cite{Gutierrez2021} and
$k=5$ \cite{Cho2022}.
%For $k=6$, analyzing the number of cases seems
%very tedious.
Conjecture C$\mathbf{(k,k)}$ has been proved
for $k \leq 7$, but with some conditions.
For the sake of completeness, we 
state exactly what has been proved for this conjecture in the next three propositions.

\begin{proposition}[{\cite[Theorem 1.2]{Gutierrez2021}}]\label{prop:2conn2vertices}
Every pair of longest cycles in a 2-connected graph intersect in at least two vertices.
\end{proposition}

\begin{proposition}[{\cite[Theorem 1.2]{Grotschel84}}]\label{prop:k3a5}
Let $k \in \{3,\ldots,5\}$
and let $G$ be a 2-connected graph with at least $k+1$ vertices.
Let $C$ and $D$ be two longest cycles in $G$.
Then $V(C) \cap V(D)$ separates $G$.
\end{proposition}

\begin{proposition}[{\cite[Theorem 1.2]{Steward1995}}]\label{prop:k6a7}
Let $k \in \{6,7\}$
and let $G$ be a graph with  circumference at least $k+1$.
Let $C$ and $D$ be two distinct longest cycles in $G$. Then $V(C) \cap V(D)$ separates $G$.
\end{proposition}

\begin{theorem}[{\cite{Grotschel84,Gutierrez2021,Steward1995}}]\label{thm:Ckk2a7}
For $2 \leq k\leq 7$, Conjecture $C(k,k)$ is true.
\end{theorem}
\begin{proof}
When $k=2$, the proof follows by Proposition \ref{prop:2conn2vertices}. So let us assume that $k \geq 3$.
Let $G$ be a $k$-connected graph with two longest cycles $C$ and $D$.
As $G$ is $k$-connected, $\delta(G) \geq k$,
which implies $|V(G)| \geq |V(C)| \geq k+1$.
If $C=D$, then $|V(C)| \cap |V(D)| \geq k$,
so let us assume that $C \neq D$.
Then, by Propositions \ref{prop:k3a5} and \ref{prop:k6a7}, $V(C) \cap V(D)$ separates $G$. As $G$ is $k$-connected, $|V(C) \cap V(D)|\geq k$ and the proof follows.
\end{proof}

By Theorems \ref{thm:CimpliesP} and \ref{thm:Ckk2a7}, we have the next result.
\begin{theorem}
For $1 \leq k\leq 6$, Conjecture $P(k,k)$
is true. That is, in every $k$-connected graph with $k\leq 6$, every pair of longest paths intersect in at least $k$ vertices.
\end{theorem}

For an arbitrary $k$, we have the next result due to Chen \etal \cite{Chen98}.
\begin{theorem}[{\cite[Theorem 2]{Chen98}}]
\label{thm:Chen}
For any $k \geq 2$, Conjecture $C(k,ck^{3/5})$ is true,
where ${c=1/(\sqrt[3]{256}+3)^{3/5}}$.
\end{theorem}

By Theorems \ref{thm:CimpliesP} and \ref{thm:Chen}, we have the next result.
\begin{theorem}
For any $k$, Conjecture $P(k,c(k+1)^{3/5}-1)$ is true,
where $c=1/(\sqrt[3]{256}+3)^{3/5}$.
\end{theorem}
%\begin{proof}
%By Proposition \ref{}, $C(k+1,c(k+1)^{3/5}$
%is true, where $c=1/(\sqrt[3]{256}+3)^{3/5}$.
%Thus, by Theorem \ref{}, Conjecture
%$P(k, c(k+1)^{3/5}-1)$ is true
%\end{proof}

Cho \etal \cite{Cho2022} showed that $P(k,k)$ is true when $k \geq \frac{n+2}{5}$, where $n$ is the number of vertices of
the graph \footnote{If we try to be complete formal, we should add the number of vertices as a third parameter to our familes of conjectures, but we will prefer to not complicate the notation.}. In fact they proved the following stronger statement.

\begin{theorem}[{\cite[Theorem 1.4]{Cho2022}}]
For any $k \geq 2$, Conjecture $P(k,\min\{k,\ceil{\frac{8k-n-4}{3}}\})$ is true.
\end{theorem}

No similar result for $C(k,r)$ has been given
in the literature. In Section \ref{section:Ckr} we show that for any $k \geq 2$, Conjecture $C(k,\min\{k,8k-n-16\})$ is true (Theorem \ref{thm:Ck8k-n-16}).
This will imply the following two results.

\begin{theorem}
For any $k \geq 2$, Conjecture $P(k,\min\{k,8k-n-9\})$ is true.
\end{theorem}

\begin{corollary}
For any $k \geq \frac{n+16}{7}$, Conjecture $C(k,k)$ is true.
For any $k \geq \frac{n+9}{7}$, Conjecture $P(k,k)$ is true.
\end{corollary}

\section{A new result on C(k,r)}\label{section:Ckr}

Our proof rely in two well-known facts, that we state in the following propositions. The first proposition is also known as Fan lemma.

\begin{proposition}[{\cite[Proposition 9.5]{BondyM08}}]\label{prop:fan-lemma}
 Let~$G$ be a~$k$-connected graph. Let~${v \in V(G)}$
and~$S \subseteq V(G) \setminus \{v\}$.
If~$|S|\geq k$, then there exists a set of~$k$~$v$-$S$ internally disjoint paths.
Moreover, every two paths in this set have~$\{v\}$
as their intersection. 
\end{proposition}

\begin{proposition}
[{\cite[Theorems 3 and 4]{Dirac52}}]\label{prop:dirac}
If~$G$ is a 2-connected graph on~$n$ vertices with minimum degree~$k$, then~$G$ has a longest cycle of length at least~$\min\{2k,n\}$.
\end{proposition}

We are now ready to show the main result of this section.

\begin{theorem}\label{thm:Ck8k-n-16}
For any~$k\geq 2$, Conjecture~$C(k,\min\{k,8k-n-16\})$ is true.
\end{theorem}
\begin{proof}
Let~$G$ be a~$k$-connected graph on~$n$ vertices.
Let~$C$ and~$D$ be two longest cycles in~$G$.
Let~$L=|C|$ and~$X=V(C) \cap V(D)$.
Let~$R$ be a longest path in a component of 
$G-C$ or~$G-D$ with length at most two.
Furthermore, if~$R \in G-C$ ($G-D$), then we choose
$R$ as a path that maximizes~$|E(R) \cap E(D)|$
($|E(R) \cap E(C)|$).

Without loss of generality, let
us suppose that~$R \in G-C$, and let~$p$ and~$r$ be the ends of~$R$.
By Proposition \ref{prop:fan-lemma}, as~$|V(C)| \geq k$, there exists a set, of~$k$~$p$-$V(C)$ internally disjoint paths that end at different vertices of~$C$.
As at most~$|R|$ of these paths includes a vertex in~$V(R) \setminus \{p\}$, there is a set, say~$\mathcal{A}$, of at least~$(k-|R|)$~$p$-$V(C)$ internally disjoint paths that end at different vertices of~$C$ and not contain any vertex in~$V(R) \setminus \{p\}$. Analogously, there is a set, say~$\mathcal{B}$, of least~$(k-|R|)$~$r$-$V(C)$ internally disjoint paths that end at different vertices of~$D$ and not contain any vertex of in~$V(R) \setminus \{r\}$.
Let~$F$ be the set of ends of~$\mathcal{A} \cup
\mathcal{B}$ in~$C$.

We construct and auxiliary weighted graph~$C^{*}$ as follows.
The vertex set of~$C^{*}$ is~$F$.
Two vertices~$u$ and~$v$ are adjacent in 
$C^{*}$ if~$C[u,v]$ does not contain any other vertex from~$F$ besides~$u$ and~$v$.
The weight of such an edge, denote it by~$w(uv)$,
is~$|C[u,v]|$. It is clear that~$L=w(C^{*})$.

Given a vertex~$u \in V(C^{*})$, we say that~$u$
is~$A$-\textit{colored} ($B$-\textit{colored}) if~$u$ is an end of a path in~$\mathcal{A}$~($\mathcal{B}$).
If~$u$ is~$A$-colored and~$B$-colored, we say that~$u$ is \textit{bicolored}.
Moreover, we say that an edge~$uv \in E(C^{*})$ is \textit{bicolored} if some of~$u$ or~$v$ is bicolored. Let~$F'$ be the set of bicolored vertices, and let~$C'$ be the set of bicolored edges. 
The following claim is clear from the definition of~$F'$ and~$C'$.

\begin{claim}\label{claim:C'geqF'}
$|C'| \geq |F'|$.
\end{claim}

Let~$F''=F \setminus F'$ and~$C''=E(C^*) \setminus C'$.

\begin{claim}\label{claim:wuvgeq}
Let~$uv \in E(C^{*})$.
If~$uv \in C''$ then~$w(uv) \geq 2$.
If~$uv \in C'$ then~$w(uv) \geq 2+|R|$.
\end{claim}
\begin{proof}
First suppose that~$uv \in C''$. Then,
neither of~$u$ or~$v$ is bicolored.
Hence, without loss of generality, either~$u$ and~$v$ are~$A$-colored or~$u$ is~$A$-colored
and~$v$ is~$B$-colored.
For the first case, let~$P_u$ and~$P_v$ be the corresponding paths in~$\mathcal{A}$.
As~$C-C[u,v]+P_u+P_v$ is also a cycle,
we must have~$|C[u,v]| \geq 2$.
For the second case, let~$P_u$ and~$Q_v$ be the corresponding paths in~$\mathcal{A}$ and~$\mathcal{B}$, respectively.
If~$P_u$ and~$Q_v$ are disjoint, then 
$C-C[u,v]+P_u+R+Q_v$ is also a cycle, so
we must have~$|C[u,v]| \geq 2+|R| \geq 2$.
Otherwise, let~$\tilde{Q}_v$ be the shortest subpath of~$Q_v$ with~$v$ as one of its ends,
and the other end, say~$x$, in~$P_u$.
As~$u$ is not~$B$-colored, then~$x \neq u$.
Let~$\tilde{P}_u$ be the subpath of~$P_u$ with ends~$u$ and~$x$. As~$C-C[u,v]+\tilde{P}_u + \tilde{Q}_v$ is also a cycle, we must have~$|C[u,v]| \geq 2$.

Now suppose that~$uv \in E(C')$, and, without loss of generality, that~$u$ is~$A$-colored and~$v$ is
bicolored.
Let~$P_u,P_v$ and~$Q_v$ be the corresponding paths
in~$\mathcal{A}$,~$\mathcal{A}$ and~$\mathcal{B}$, respectively. If~$P_u$ and~$Q_v$ are disjoint, then 
$C-C[u,v]+P_u+R+Q_v$ is also a cycle, so
we must have~$|C[u,v]| \geq 2+|R|$.
Otherwise, let~$\tilde{Q}_v$ be the shortest subpath of~$Q_v$ with~$r$ as one of its ends,
and the other end, say~$x$ in~$P_u \cup P_v$.
If~$x \in P_u$, let~$\tilde{P}_u$ be the subpath of~$P_u$ with ends~$u$ and~$x$.
As~$C-C[u,v]+\tilde{P}_u + \tilde{Q}_v+R +P_v$ is also a cycle, we must have~$|C[u,v]| \geq 3+|R|$.
Otherwise, let~$\tilde{P}_v$ be the subpath of~$P_v$ with ends~$v$ and~$x$.
As~$C-C[u,v]+\tilde{P}_v + \tilde{Q}_v+R +P_u$ is also a cycle, we must have~$|C[u,v]| \geq 3+|R|$.
\end{proof}

\begin{claim}
$|E(C^*)|=|\mathcal{A}|+|\mathcal{B}|-|F'|$.
\end{claim}
\begin{proof}
Note that~$|E(C^*)|=|F|$.
Let~$F_a$ be the set of~$A$-colored vertices, and let~$F_b$ be the set of~$B$-colored vertices.
Then~$|F|=|F_a| + |F_b| -|F_a \cap F_b|=|\mathcal{A}|+|\mathcal{B}|-|F'|$.
\end{proof}

\begin{claim}\label{claim:Lgeq4(k-2)}
$L \geq 4(k-|R|)+|R||C'|-2|F'|$.
\end{claim}
\begin{proof}
By Claim \ref{claim:wuvgeq},
$w(C')\geq (2+|R|)|C'|$
and~$w(C'') \geq 2|C''|$.
So
\begin{eqnarray*}
L        &=&w(C')+w(C'') \\
         &\geq& (2+|R|)|C'|+2|C''| \\
         &=&2(|C'|+|C''|)+|R||C'|\\
         &=&2|E(C^*)|+|R||C'|\\
         &=&2|\mathcal{A}|+2|\mathcal{B}|-2|F'|+|R||C'|\\
         &\geq& 4(k-|R|)+|R||C'|-2|F'|,
\end{eqnarray*}
as we want.
\end{proof}

By Claim \ref{claim:Lgeq4(k-2)},
if $|R|=2$, then $L \geq 4(k-2)$.
Now, by Proposition \ref{claim:Lgeq4(k-2)},~$|X|\geq 2L-n$. Thus,~$|X| \geq 8k-n-16$, as we want.
Hence, from now on, we may assume that 
$|R|\leq 1$.
This condition will also imply that any component of~$D-C$ has at most two vertices.
Thus, every component of~$D-C$ is either an edge or a vertex, and
\begin{equation}\label{eq:3XgeqL}
L=|V(C) \cap V(D)|+ |V(D) \setminus V(C)| \leq |X|+2|X| = 3|X|.
\end{equation}

Suppose for a moment that every component in~$D-C$ 
is a vertex. By a similar reasoning to the previous paragraph,~$L\leq 2|X|$. 
By Proposition \ref{prop:dirac},~$L \geq \min\{2k,n\}$.
If~$\min\{2k,n\}=n$, then~$L=n$ and~$G$ is hamiltonian, so~$|X|=n=k$ and we are done;
otherwise,~$|X|\geq L/2 \geq k$ and we
are also done. 
Hence, there exists an edge in $D-C$ and, by the choose of~$R$,~$R$ is an edge
in~$D-C$.
From now on,
if $uv \in E(C^{*})$ has a vertex in $X$ we say that $uv$ is \textit{covered}.

\begin{claim}\label{claim:bicolimplcovmoreover}
Let~$u \in V(C^{*})$. 
If~$u$ is bicolored, then it is also covered.
Moreover, if~$v$ is a neighbor of~$u$ in~$C^{*}$,
then~$w(uv) \geq 3$.
\end{claim}
\begin{proof}

Let~$P_u$ and~$Q_u$ be the corresponding paths
in~$\mathcal{A}$, and~$\mathcal{B}$, respectively. 
If~$|P_u|>1$, then there exists a path of length two in~$G-C$, a contradiction to the choise of $R$.
Hence~$|P_u|=1$, and, analogously,~$|Q_u|=1$.
Suppose by contradiction that~$u$ is not covered.
In that case,~$D-R+P_u+Q_u$ is a cycle
longer than~$L$, a contradiction. This proves the first part of the claim.
For the second part,
suppose, without loss of generality, that~$v$ is~$A$-colored,
and let~$P_v \in \mathcal{A}$
the corresponding path.
As before, we can show that~$|P_v|=1$. Hence,
as~$C-C[u,v]+Q_u+R+P_v$ is a cycle,
we must have~$|C[u,v]|\geq 3$.
\end{proof}

Suppose for a moment that~$|F'|=0$.
In that case,~$|C'|=0$, so,
by Claim \ref{claim:Lgeq4(k-2)},
$L \geq 4(k-1)$.
If~$k\leq 3$, then the proof follows by
Theorem \ref{thm:Ckk2a7}. Hence,
we may assume that~$k\geq 4$, which implies
that~$L \geq 3k$, and, as~$3X\geq L$, the proof follows.
Thus, from now on, we may assume that~$|F'|>0$.

\begin{claim}\label{claim:C'>F'}
If $|C'| = |F'|$ then $|X| \geq k$.
\end{claim}
\begin{proof}
Suppose that~$|C'|=|F'|$.
In that case, as~$|F'|>0$, every vertex of~$F$ is bicolored, so~$F=F'$.
As $|R|=1$, we have~$|\mathcal{A}|\in \{k-1,k\}$.
If $|\mathcal{A}|=k$ then, as every vertex of $F$ is bicolored, we have $|F|=|\mathcal{A}|=k$. By Claim \ref{claim:bicolimplcovmoreover}, $|X| \geq |F|=k$ and the proof follows.
Hence, from now on, we may assume that
$|\mathcal{A}|=k-1$.
Then, there exists
a~$p$-$V(C)$ path, say~$\tilde{P}$, that contains~$r$ and
does not end in any vertex of~$F$.
Let~$w$ be the other end of~$\tilde{P}$
and let~$\tilde{P}_{qw}$ be the subpath of~$\tilde{P}$ with ends
$r$ and~$w$.
Let~$u$ and~$v$ be the vertices of~$F$
such that~$r\in C[u,v]$.
Let $P_u$ and $Q_v$ be the corresponding paths
in $\mathcal{A}$ and $\mathcal{B}$, respectively.
Then, as $C-C[u,w]+P_u+R+ \tilde{P}_{qw}$ is
a path, $|C[u,w]|\geq 3$
So~$|C[u,v]| \geq 4$ and there exists a path in~$C-D$ with size at least 2, a contradiction to the choice of $R$.
\end{proof}

By Claim \ref{claim:bicolimplcovmoreover},~$|X| \geq |C'|$ and~$w(C')\geq 3|C'|$. By Claim \ref{claim:wuvgeq},~$w(C'') \geq 2|C''|$.
Also, we may assume that $|F'| \leq |C'|-1$ by Claim  \ref{claim:C'>F'}.
So, as $|R|=1$, by Claim \ref{claim:Lgeq4(k-2)} we have that
\begin{eqnarray*}
L       &\geq& 4(k-1)+|C'|-2|F'| \\
 &\geq& 4k-4+2-|C'| \\
  &\geq& 4k-4+2-|X|.
\end{eqnarray*}
As~$3|X| \geq L$ by \eqref{eq:3XgeqL}, we have
$|X| \geq k-1/2$, which finishes the proof.
\end{proof}

\section{Concluding remarks} \label{section:conclusion}

In this paper, we show that every pair of longest cycles in a~$k$-connected graph intersect each other in at least~$\min\{k,8k-n-16\}$ vertices,
and that every pair of longest paths intersect each other in at least~$\min\{k,8k-n-9\}$ vertices. A direct corollary of these results is that,
if~$k \geq (n+16)/7$, then every pair of longest cycles intersect in at least~$k$ vertices and, if~$k \geq (n+9)/7$, then every pair of longest paths intersect in at least~$k$ vertices.
In \cite{Gutierrez2021-Pro}, a set of conjectures, called CHC($r$), affirm that
if~$k \geq n/r$, then every pair of longest paths intersect in at least~$k$ vertices.
In this paper, we progress towards the case~$r=7$. We believe the techniques presented here can be used for proving the cases when~$r>7$.

\bibliographystyle{amsplain}
\bibliography{bibliografia}

\end{document}